\documentclass[12pt]{article}
\usepackage[utf8]{inputenc}
\usepackage{geometry}
\usepackage{amsmath, amsthm, amssymb} % Equations
\usepackage{savesym}
\savesymbol{P}
\usepackage{graphicx}
\usepackage[dvipsnames]{xcolor}
\graphicspath{ {./images/} }
\usepackage{hyperref}
\hypersetup{
    colorlinks=true,
    linkcolor=magenta,
    filecolor=magenta,      
    urlcolor=magenta,
    citecolor=magenta
} % change these colors if you like!
\newcommand{\Z}{\mathbb{Z}}

\newcommand{\N}{\mathbb{N}}
\newcommand{\R}{\mathbb{R}}
\newcommand{\C}{\mathbb{C}}
\newcommand{\Harm}{\mathcal{H}}
\newcommand{\diag}{\text{diag}}
\newcommand\scalemath[2]{\scalebox{#1}{\mbox{\ensuremath{\displaystyle #2}}}}

\usepackage{tikz}
\theoremstyle{definition}

\newtheorem{proposition}{Proposition}
\newtheorem{lemma}{Lemma}
\newtheorem{theorem}{Theorem}
\newtheorem{corollary}{Corollary}
\theoremstyle{remark}

\begin{document}

%\maketitle
\begin{center}
    \begin{Large} Graded multiplicity in harmonic polynomials from the Vinberg setting \end{Large}\\
    Alexander Heaton\\
    \today{}
\end{center}

\begin{abstract}
    %We consider a specific family of examples with $K=GL_2 \times \cdots \times GL_2$ related to a cyclic quiver and falling into the following context (first considered by Vinberg): Let G be a connected reductive algebraic group over the complex numbers. A subgroup, K, of fixed points of a finite-order automorphism acts on the Lie algebra of G.  Each eigenspace of the automorphism is a representation of K.  Let $\mathfrak{g}_1$ be one of the eigenspaces.  We consider the harmonic polynomials on $\mathfrak{g}_1$ as a representation of K, which is graded by homogeneous degree. Given any irreducible representation of K, we will see how its multiplicity in the harmonic polynomials is distributed among the various graded components. The results are described geometrically by counting integral points on certain faces of a polyhedron.

    We consider Vinberg $\theta$-groups associated to a cyclic quiver on $r$ nodes. Let $K$ be the product of general linear groups associated to the nodes, acting naturally on $V = \oplus \text{Hom}(V_i, V_{i+1})$. We study the harmonic polynomials on $V$ in the specific case where $\dim V_i = 2$ for all $i$. For each multigraded component of the harmonics, we give an explicit decomposition into irreducible representations of $K$, and additionally describe the multiplicities of each irreducible by counting integral points on certain faces of a polyhedron.
\end{abstract}

\section{Introduction} \label{Sec:introduction}

Consider the representations of a cyclic quiver on $r$ nodes. If $r=12$ we have
\begin{center}
\begin{tikzpicture}
\def \n {12}
\def \radius {2.5cm}
\def \margin {8} % margin in angles, depends on the radius
% Thanks to: Jerome Tremblay
\foreach \s in {1,...,\n}
{
  \node[draw, circle] at ({-360/\n * (\s+9)}:\radius) {$\s$};
  \draw[<-, >=latex] ({360/\n * (\s-1)+\margin}:\radius)
    arc ({360/\n * (\s -1)+\margin}:{360/\n * (\s)-\margin}:\radius);
}
\end{tikzpicture}
\end{center}
For each node $j$, associate a finite-dimensional vector space $V_j$. For each arrow $j \rightarrow j+1$ (mod $r$), associate the space of linear transformations, $\mbox{Hom}(V_j, V_{j+1})$.  Set $V = V_1 \oplus \cdots \oplus V_r$ and let $K$ be the block diagonal subgroup of $G=GL(V)$ isomorphic to $GL(V_1) \times \cdots \times GL(V_r)$ acting on
\begin{equation*}
    \mathfrak p = \mbox{Hom}(V_1,V_2) \oplus \mbox{Hom}(V_2,V_3) \oplus \cdots \oplus \mbox{Hom}(V_{r-1}, V_r) \oplus \mbox{Hom}(V_r, V_1).
\end{equation*}
Here we let $GL(U) \times GL(W)$ act on $\mbox{Hom}(U,W)$ by $(g_1,g_2) \cdot T  = g_2 \circ T \circ g_1^{-1}$, as usual. For $(T_1, \dots, T_r) \in \mathfrak p$, we have $K$-invariant functions defined by
\begin{equation*}
    %tr_p(T_1,\dots,T_r) = \text{Trace}\left[ (T_1 \circ \cdots \circ T_r)^p \right],
    \text{Trace}\left[ (T_1 \circ \cdots \circ T_r)^p \right]
\end{equation*}
for $1 \leq p \leq \min \{\dim V_j \}$. By a result of Le Bruyn and Procesi \cite{LP90}, these generate the $K$-invariant polynomial functions on $\mathfrak p$.

\color{black} \textbf{Main results:} In this paper, we study the case $\dim V_j = 2$ for all $j \in \{1,\dots,r\}$, and consider the representation of $K = GL_2 \times \cdots \times GL_2$ and also $K = S(GL_2 \times \cdots \times GL_2)$ on the $K$-harmonic polynomials on $\mathfrak{p}$, denoted $\mathcal{H} \subset \C[\mathfrak{p}]$. Theorem \ref{thm:harmonics-character} gives an explicit decomposition of the multigraded component $\mathcal{H}_n$ for $n \in \N^r$, while Theorem \ref{thm:geometric-multiplicity} describes the multiplicity of any particular irreducible representation of $K$ inside $\mathcal{H}_n$ by counting integer points on the intersection of certain polyhedra. Figure \ref{fig:r=3-example-multiplicity-six} displays an example with $r=3$. 

In a recent paper \cite{frohmader-heaton} with Frohmader, we solve the general case $K=GL_{k_1} \times \cdots \times GL_{k_r}$ differently, by summing over certain distinguished tableau, and with results only valid for a stable range of parameters. In particular, the results covered in the present paper fall outside the stable range, and hence are not covered by results from \cite{frohmader-heaton}. The results of \cite{frohmader-heaton} use a branching rule from \cite{branchingrule}, which introduces the stable range restriction. We hope these results may be extended outside the stable range by developing another branching rule using crystal bases, and eventually recover the examples of the present paper as well. \color{black}

Throughout the paper, the ground field is $\C$. Leaving the cyclic quiver for a moment, we recall the the definition of $G$-harmonic polynomials. Let $G$ denote a linear algebraic group. Given a regular representation $V$ of $G$, we denote the algebra of polynomial functions on $V$ by $\C[V]$, defined by identifying with $Sym(V^*)$, the algebra of symmetric tensors on the dual of $V$.

The constant coefficient differential operators on $\C[V]$ will be denoted by $\mathcal{D}(V)$, identified with $Sym(V)$.  The differential operators without a constant term will be denoted $\mathcal{D}(V)_+$ and the $G$-invariant differential operators are $\mathcal{D}(V)^G$. Let
\begin{equation*}
    \mathcal{H}(V) = \left \{ f \in \C[V] : \Delta f = 0 \mbox{ for all } \Delta \in \mathcal{D}(V)^{G}_+ \right \}
\end{equation*}
be the $G$-\emph{harmonic} polynomial functions. In the case of $G=SO_3$ acting on its defining representation, $\mathcal{D}(V)^G_+$ is generated by the Laplacian $\partial_x^2 + \partial_y^2 + \partial_z^2$ and the harmonics decompose into minimal invariant subspaces, which, when restricted to the sphere, admit an orthogonal basis, namely the Laplace spherical harmonics familiar from physics. In that case, decomposing the harmonic polynomials (the subject of this paper) leads to a complete set of orthogonal functions on the sphere, useful in numerous theoretical and practical applications. For a nice exposition from this perspective, see \cite{V89}.

In general, every polynomial function can be expressed as a sum of $G$-invariant functions multiplied by $G$-harmonic functions.  That is, there is a surjection
\begin{equation*}
    \C[V]^G \otimes \mathcal{H}(V) \rightarrow \C[V] \rightarrow 0
\end{equation*}
obtained by linearly extending multiplication.  
For the Laplace spherical harmonics, this is an isomorphism. All invariants are generated by the squared Euclidean distance function, and any polynomial can be written uniquely as a product of its radial and spherical components.

The scalar multiplication of $\C$ on $V$ commutes with the action of $G$.  The resulting $\C^\times$ action gives rise to a gradation on $\C[V]$, which is the usual notion of \emph{degree}.  The $G$-harmonic functions inherit this gradation, so we define $\mathcal{H}_n(V)$ as the homogeneous $G$-harmonic functions of degree $n$. We have the direct sum of $G$-representations
\begin{equation*}
    \mathcal{H}(V) = \bigoplus_{n=0}^\infty \mathcal{H}_n(V).
\end{equation*}
For a reductive linear algebraic group $G$, every regular representation is completely reducible.  Let $\{F_\mu\}_{\mu \in \widehat{G}}$ denote a set of representatives of the irreducible representations of $G$.\\
\textbf{Problem:} For each $n$, how does $\mathcal{H}_n(V)$ decompose?  That is, given $\mu \in \widehat G$, what is the multiplicity of $F_\mu$ inside $\mathcal{H}_{n}$, denoted
\begin{equation*}
    \text{dim Hom}( F_\mu, \mathcal{H}_n(V) ) = \mbox{ ?}
\end{equation*}
Returning to the cyclic quiver above, the $K$-harmonic functions on $\mathfrak{p}$ form a graded representation of $K$:
\begin{equation*}
    \mathcal{H}(\mathfrak{p}) = \bigoplus_{n=0}^\infty \mathcal{H}_n(\mathfrak{p})
\end{equation*}
The fact that the polynomial functions are a free module over the invariants is a consequence of the Vinberg theory of $\theta$-groups \cite{V76}.  Yet, the literature on quivers does not seem to address the structure of the associated harmonic polynomials.

\subsection{Background from some existing literature}\label{subsection:background-literature}

As a representation of $K$, the harmonics are equivalent to an induced representation.  For details, see Chapter 3 of \emph{Geometric invariant theory over the real and complex numbers} \cite{W2017}. Alternatively, see \emph{An Analogue of the Kostant-Rallis Multiplicity Theorem for $\theta$-group Harmonics} \cite{W17}, which also describes the multiplicities, but ignoring the gradation.

The standard results concerning spherical harmonics on $\mathbb R^3$ were generalized by Kostant in \emph{Lie group representations on polynomial rings} \cite{K63}.  This is Kostant's most often cited paper. Among its many results, it establishes that $\C[\mathfrak{g}]$ is a free module over $\C[\mathfrak{g}]^G$ for a connected reductive group $G$.

To a combinatorialist, it is natural to consider the polynomial defined by the series
\begin{equation*}
    p_\mu(q) = \sum_{n=0}^\infty \text{dim Hom}( F_\mu, \mathcal{H}_n(V) ) \,\, q^n.
\end{equation*}
In the case addressed by Kostant, $V = \mathfrak{g}$, these polynomials extract deep information in representation theory.  For starters, they are Kazhdan-Lusztig polynomials for the affine Weyl group \cite{K82}.  Outside of Kostant's setting, very little is known about them.

In the case that $\mathfrak{g}$ is of Lie type A, then $p_\mu (q)$ was studied by Stanley in \cite{S84}.  Later on, connections with Hall-Littlewood polynomials were made \cite{M95}.  Even combinatorial interpretations for their coefficients are known.  An alternating sum formula was found by Hesselink in \cite{H80}.

In 1971, Kostant and Rallis obtained a generalization applying to symmetric pairs $(G,K)$ \cite{KR71}. That is, $K$ is the fixed point set of a regular involution on a connected reductive group $G$.  A natural way to generalize is to consider $K$ that are fixed by automorphisms of order larger than two.  Exactly this was done by Vinberg in his 1976 theory of $\theta$-groups, published as \emph{The Weyl group of a graded Lie algebra} \cite{V76}.  Since then, an enormous amount of work has been done on $\theta$-groups, but the analog of the graded structure of harmonic polynomials still does not exist. Taking $\theta:G \rightarrow G$ to be the identity automorphism on $G=SO_3$ we have $K=G$ acting on its Lie algebra $\mathfrak{g} \cong \R^3$ and recover the spherical harmonics example.

\section{Preliminary setup and main results}\label{section:setup-notation-and-main-results}\color{black}

In this section we setup notation, precisely state the main Theorems \ref{thm:harmonics-character} and \ref{thm:geometric-multiplicity}, proving Theorem \ref{thm:geometric-multiplicity} as a corollary of Theorem \ref{thm:harmonics-character}. In Section \ref{sec:proofs} we complete the proof of Theorem \ref{thm:harmonics-character}.

Throughout the paper we use the notation $[r] = \{1,2,\dots,r\}$ for any $r \in \N_{>0}$, where $\N = \{0,1,2,\dots\}$. For $k \in \Z$ we also use $[k]_2 = 0$ if $k$ is even and $[k]_2 = 1$ if $k$ is odd. If $n \in \N^r$, we use indices mod $r$ with representatives from $[r]$ so that, for example, the expression $n_i - n_{i-1}$ equals $n_1 - n_r$ when $i=1$, and $n_i$ denotes the $i$th component of $n \in \N^r$ as usual. We also denote the all-ones vector as $e = (1,\dots,1) \in \N^r$.

Let $\C_k$ denote the irreducible representation of $\C^\times$ on $\C$ given by $v \mapsto w^k v$ for $w \in \C^\times, v \in \C$ and $k \in \Z$. Let $F_k$ denote the irreducible representation of $SL_2(\C)$ on $\text{Sym}^k(\C^2)$ for $k \in \N$. For $z \in \Z^k$ and $s \in \N^\ell$ let $F_{z,s}$ denote the outer tensor product representation
\begin{equation}\label{eqn:Ktilde-irreps}
    F_{z,s} = \C_{z_1} \otimes \cdots \otimes \C_{z_k} \otimes F_{s_1} \otimes \cdots \otimes F_{s_\ell}
\end{equation}
of $(\C^\times)^k \times SL_2^\ell$, which is irreducible by \cite[Proposition 4.2.5]{GoodmanWallach}. Let $\diag(a_1,\dots,a_k)$ denote the diagonal matrix with entries $a_1,\dots,a_k$, and recall that for $w \in \C^\times$, $\diag(w, w^{-1}) \in SL_2$ acts on $F_k$ by $\diag(w^k, w^{k-2}, \dots, w^{-k})$ under a suitable choice of basis.

\color{black}
We now describe the infinite family of representations $V=V_r$ for which Theorems \ref{thm:harmonics-character} and \ref{thm:geometric-multiplicity} apply. For each $r \in \N_{>1}$, let $G=GL_{2r}$ or $G=SL_{2r}$. The results for the two groups will differ only slightly, and we will point out these differences as they arise. If $\theta:G\to G$ is any inner automorphism of order $r$, then $\theta(g) = hgh^{-1}$ for some $h \in G$. Since $\theta^r=id$, then $h^r g h^{-r} = g, \, \forall g$, and $h^r = \lambda I$ for some $\lambda \in \C^\times$. Since we are interested in the fixed point subgroup $K = G^\theta = \{ k \in G: hkh^{-1} = k \}$ we can take $\lambda = 1$ without loss of generality. Then the eigenvalues of $h$ are the $r$th roots of unity appearing with some multiplicities. Each eigenvalue of multiplicity $1$ will produce a factor of $\C^\times$ in $K=G^\theta$ and so this article deals with the first interesting case, where all multiplicities are $2$.

Therefore we take $\theta(g) = hgh^{-1}$ where $h$ is the diagonal matrix whose eigenvalues are the $r$th roots of unity, each appearing with multiplicity $2$. The fixed point subgroup $K = G^\theta$ consists of the block-diagonal subgroup of $G$
\begin{equation*}
    K = \underbrace{GL_2 \times \cdots \times GL_2}_{r \text{ factors}} \hspace{1cm} \text{ or } \hspace{1cm} K = S\big( \underbrace{GL_2 \times \cdots \times GL_2}_{r \text{ factors}} \big),
\end{equation*}
where the $S$ indicates requiring overall determinant $1$, depending on if $G=GL_{2r}$ or $G=SL_{2r}$, respectively. \color{black} In other words, for $r \in \N_{>1}$, we consider $K = GL_2^r$ or $K = S(GL_2^r)$, which is the set of all tuples $(g_1,\dots,g_r)$ with $g_j \in GL_2$ for all $j \in [r]$ but $\prod_j \det g_j = 1$ in the case $K=S(GL_2^r)$. \color{black}

The Lie algebra $\mathfrak{g} = \text{Lie}(G)$ breaks up as $\mathfrak{g} = \oplus \mathfrak{g}_k$ where $\mathfrak{g}_k$ is the $e^{2\pi i k/r}$ eigenspace of the differential $d\theta_1:\mathfrak{g}\to \mathfrak{g}$. Then $K$ acts on each eigenspace by restricting the Adjoint action of $G$ on $\mathfrak{g}$. For the eigenspace $\text{Lie}(K) = \mathfrak{g}_0$, then $K$ acts on its Lie algebra and we have the original setting of Kostant in \cite{K63}. \color{black} For each $r \in \N_{>1}$, let $V = V_r = \mathfrak{g}_1$, the $e^{2\pi i/r}$ eigenspace. \color{black} As can easily be checked, $V = \oplus_{j \in [r]} M_j$ where each $M_j \simeq \text{Hom}(\C^2, \C^2)$ and the action of $g=(g_1,\dots,g_r) \in K = G^\theta$ on $(X_1,\dots,X_r) \in V$ is given by
\begin{equation}\label{eqn:action-by-cyclic-conjugation}
    X_i \mapsto g_i X_i g_{i+1}^{-1}
\end{equation}
where indices are taken mod $r$ from $[r] = \{1,2,\dots,r\}$. This recovers the action on $\mathfrak{p}$, the representations of the cyclic quiver described in the introduction.

\color{black}
Let $\C[V]$ be the ring of polynomial functions on the $e^{2\pi i/r}$ eigenspace $V = \mathfrak{g}_1$. Since $K$ acts on $V$, then $K$ acts on $\C[V]$ by $gf(x) = f(g^{-1}x)$ for $g \in K, f \in \C[V]$. Since these examples fall under Vinberg's theory \cite{V76} we have
\begin{equation*}
    \C[V] = \C[V]^K \otimes \mathcal{H}.
\end{equation*}
By a result in \cite{LP90} we know that $\C[V]^K$ is a polynomial algebra generated by $tr(X_1 X_2 \cdots X_r)$ and $tr((X_1 X_2 \cdots X_r)^2)$, a fact we will use to prove Theorem \ref{thm:harmonics-character}. 

Since $V=\oplus_{j \in [r]} M_j$ for $M_j \simeq \text{Hom}(\C^2,\C^2)$ there is a natural multigradation
\begin{equation*}
    \C[V] = \bigoplus_{n \in \N^r} \C[V]_n
\end{equation*}
by multihomogeneous degree $n \in \N^r$, where the polynomials in $\C[V]_n$ have homogeneous degree $n_i$ in the variables associated with $M_i$. With $\Harm_n = \Harm \cap \C[V]_n$ we have a direct sum of $K$ representations
\begin{equation*}
    \mathcal{H} = \bigoplus_{n \in \N^r} \mathcal{H}_n.
\end{equation*}
\color{black}
Theorem \ref{thm:harmonics-character} will decompose $\Harm_n$ into irreducible representations of $K$, while Theorem \ref{thm:geometric-multiplicity} will describe the multiplicity of each irreducible representation in $\Harm_n$ by counting integral points on certain polyhedra.

\color{black}
First we describe the irreducible representations of $K$. If $K = S(GL_2^r)$, let $\widetilde{K} = (\C^\times)^{r-1} \times SL_2^r$ and let $\varphi:\widetilde{K} \to K$ be the surjective homomorphism
\begin{equation*}
    (w_1,\dots,w_{r-1},g_1,\dots,g_r) \mapsto (w_1g_1,\dots, w_{r-1}g_{r-1}, w_1^{-1}w_2^{-1}\cdots w_{r-1}^{-1} g_r)
\end{equation*}
where $w_i \in \C^\times$ for $i \in [r-1]$ and $g_i \in SL_2$ for $i \in [r]$. Then $N = \ker \varphi \simeq \{1,-1\}^{r-1}$ is a finite group with $2^{r-1}$ elements
\begin{equation*}
    \big\{ (w_1,\dots, w_{r-1}, \begin{bmatrix}
     w_1 & 0\\ 0 & w_1
    \end{bmatrix}, \dots, \begin{bmatrix}
     w_{r-1} & 0\\ 0 & w_{r-1}
    \end{bmatrix}, \begin{bmatrix}
     \prod w_i & 0\\ 0 & \prod w_i
    \end{bmatrix}) : w_i \in \{1,-1\}, i \in [r-1] \big\}.
\end{equation*}
If $K=GL_2^r$, let $\Tilde{K} = (\C^\times)^r \times SL_2^r$ and let $\varphi:\widetilde{K} \to K$ be the surjective homomorphism
\begin{equation*}
    (w_1,\dots,w_r,g_1,\dots,g_r) \mapsto (w_1g_1,\dots, w_{r}g_{r}).
\end{equation*}
Then $N = \ker \varphi \simeq \{1,-1\}^r$ is a finite group of $2^r$ elements
\begin{equation*}
    \big\{ (w_1,\dots, w_{r}, \begin{bmatrix}
     w_1 & 0\\ 0 & w_1
    \end{bmatrix}, \dots, \begin{bmatrix}
     w_{r} & 0\\ 0 & w_{r}
    \end{bmatrix}) : w_i \in \{1,-1\}, i \in [r] \big\}.
\end{equation*}

Every representation of $K$ yields a representation of $\Tilde{K}$ upon composition with $\varphi$, so the irreducible representations of $K$ are those irreducible representations of $\Tilde{K}$ which factor through $\varphi$, so that all elements of the finite group $N = \ker \varphi$ act by the identity transformation. If the representation $(\sigma_{z,s}, F_{z,s})$ of $\Tilde{K}$ has the property that $\sigma_{z,s}(w) = I$ for every $w \in N$, then $(\sigma_{z,s}, F_{z,s})$ is also a representation of $K \simeq \Tilde{K}/N$ and the coset $kN \in K$ acts by $\sigma_{z,s}(k):F_{z,s} \to F_{z,s}$ for $k \in \Tilde{K}$.

\begin{proposition}\label{prop:parametrize-irreps-mod-2} \hspace{1cm}
\begin{enumerate}
    \item Up to isomorphism, the irreducible representations of $K = GL_2^r$ are the $F_{z,s}$ of Equation (\ref{eqn:Ktilde-irreps}) for $z \in \Z^r$ and $s \in \N^r$, subject to the condition that $z_j + s_j$ is even for all $j \in [r]$.
    \item Up to isomorphism, the irreducible representations of $K = S(GL_2^r)$ are the $F_{z,s}$ of Equation (\ref{eqn:Ktilde-irreps}) for $z \in \Z^r$ and $s \in \N^r$, subject to the condition that $z_j + s_j + s_r$ is even for all $j \in [r]$.
\end{enumerate}
\end{proposition}

\begin{proof}
First consider $K = S(GL_2^r)$. Let $\sigma_{z,s}(w):F_{z,s} \to F_{z,s}$ denote the representation of $w \in N$ on $F_{z,s}$ of Equation (\ref{eqn:Ktilde-irreps}). Recall that $\diag(a, a^{-1}) \in SL_2$ acts on $F_k$ by $\diag(a^k, a^{k-2}, \dots, a^{-k})$ under a suitable choice of basis. If $a = \pm 1$, we have $\pm I$ depending on the parity of $k$. Since $w \in N$ has each $w_i = \pm 1$ for $i \in [r-1]$, we have 
\begin{align*}
    \sigma_{z,s}(w) &= \bigotimes_{j=1}^{r-1} w_j^{z_j} I \otimes \bigotimes_{j=1}^{r-1} w_j^{s_j} I \otimes \left(\prod_{j=1}^{r-1} w_j \right)^{s_r} I \\
    &= \prod_{j=1}^{r-1} w_j^{z_j + s_j + s_r} \,\, \bigotimes_{j=1}^{2r - 1} I,
\end{align*}
where $I$ is the identity operator on the corresponding factor of the tensor product $F_{z,s}$ in Equation (\ref{eqn:Ktilde-irreps}). Thus, if $z_j + s_j + s_r$ is even, then $\sigma_{z,s}(w)$ acts as the identity for all $w \in N$. Conversely, if $\sigma_{z,s}(w)$ acts as the identity for all $w \in N$, then taking $w$ as the element with $-1$ in the $j$th coordinate and ones elsewhere, we see that $z_j + s_j + s_r$ must be even from the formula above.

Now consider $K = GL_2^r$. In that case,
\begin{align*}
    \sigma_{z,s}(w) &= \bigotimes_{j=1}^{r} w_j^{z_j} I \otimes \bigotimes_{j=1}^{r} w_j^{s_j} I \\
    &= \prod_{j=1}^{r} w_j^{z_j + s_j} \,\, \bigotimes_{j=1}^{2r} I,
\end{align*}
and we see that $\sigma_{z,s}(w)$ is the identity for every $w \in N$ exactly when $z_j + s_j$ is even for all $j \in [r]$, similarly.
\end{proof}
\color{black}

Recall that $K$ acts on $V$ by Equation (\ref{eqn:action-by-cyclic-conjugation}), and hence acts on the $K$-harmonic polynomial functions $\mathcal{H} \subset \C[V]$, with each multigraded component $\Harm_n$ an invariant subspace, for $n \in \N^r$. We can now state our first main result.

\begin{theorem}\label{thm:harmonics-character}
Let $G = SL_{2r}$ and $K = G^\theta = S(GL_2^r)$.
For any $n \in \N^r$, let $z \in \Z^{r-1}$ have components $z_i = n_i - n_{i-1} - n_r + n_{r-1}$ for $i \in [r-1]$. Then the multigraded component $\mathcal{H}_n$ decomposes into irreducible representations of $K$ as
\begin{align}\label{eq:harmonics-character}
    \mathcal{H}_n = \bigoplus_{m \in \partial \Lambda_n} \bigoplus_{s \in \partial \lambda_m} F_{z,s}
\end{align}
where $\Lambda_n$ and $\lambda_m$ denote the sets
\begin{align*}
    \Lambda_n &= \prod_{i \in [r]} \{ n_i, n_i - 2, \dots, [n_i]_2 \} \subset \N^r \text{ and }\\
    \lambda_m &= \prod_{i \in [r]} \{ m_i + m_{i-1}, m_i + m_{i-1} - 2, \dots, |m_i - m_{i+1}| \} \subset \N^r,
\end{align*}
while $\partial \Lambda_n = \{ x \in \Lambda_n : x + 2e \notin \Lambda_n \}$ and $\partial \lambda_m = \{ x \in \lambda_m : x + 2e \notin \lambda_m \}$.

If instead $G=GL_{2r}$ and $K = GL_2^r$, then for $n \in \N^r$ let $z \in \Z^r$ have components $z_i = n_i - n_{i-1}$ for $i \in [r]$. In that case, Equation (\ref{eq:harmonics-character}) again records the decomposition of $\Harm_n$ into irreducible representations of $K$.
\end{theorem}

We will prove Theorem \ref{thm:harmonics-character} in Section \ref{sec:proofs}. In the remainder of this section, we give an example of the sets $\Lambda_n, \lambda_m, \partial \Lambda_n, \partial \lambda_m$, and then derive Theorem \ref{thm:geometric-multiplicity} as a consequence of Theorem~\ref{thm:harmonics-character}. \color{black}
With $r=3$, $n=(3,2,3)$, and $m=(3,2,1)$ we have
\begin{align*}
    \nonumber \Lambda_n &= \{ 3,1 \} \times \{ 2,0 \} \times \{ 3,1 \} \\
     &= \{ (3,2,3),(3,2,1),(3,0,3),(3,0,1),(1,2,3),(1,2,1),(1,0,3),(1,0,1) \} \\
    \lambda_m &= \{4,2 \} \times \{5,3,1 \} \times \{3,1 \} \\
     &= \big\{ (4,5,3), (4,5,1), (4,3,3), (4,3,1), (4,1,3), (4,1,1),\\
     &\hspace{9mm} (2,5,3), (2,5,1), (2,3,3), (2,3,1), (2,1,3), (2,1,1) \big\}\\
    \partial \Lambda_n &= \Lambda_n \setminus \{ (1,0,1) \} \\
    \partial \lambda_m &= \lambda_m \setminus \{ (2,3,1), (2,1,1) \}
\end{align*}
\color{black} Notice that each $F_{z,s}$ appearing in the decomposition (\ref{eq:harmonics-character}) satisfies the criteria in Proposition \ref{prop:parametrize-irreps-mod-2}. We check this explicitly, and what follows applies for all $j \in [r]$. For $a,b \in \Z$ let $a \equiv b$ if $a$ and $b$ are both even, or both odd. In other words, $a \equiv b$ denotes equivalence mod $2$, so that $a$ and $b$ have the same parity. Since $s \in \partial \lambda_m$ implies that $s_j \equiv m_j + m_{j-1}$ and $m \in \partial \Lambda_n$ implies that $m_j \equiv n_j$, we see that $s_j \equiv n_j + n_{j-1}$. For $K = GL_2^r$, Theorem \ref{thm:harmonics-character} states that $z_j = n_j - n_{j-1}$. Thus $z_j \equiv n_j + n_{j-1} \equiv s_j$ and $s_j$ and $z_j$ have the same parity, and hence their sum is even, as required by Proposition \ref{prop:parametrize-irreps-mod-2}. For $K = S(GL_2^r)$, Theorem \ref{thm:harmonics-character} states that $z_j = n_j - n_{j-1} - n_r + n_{r-1} \equiv n_j + n_{j-1} + n_r + n_{r-1}$. But then $z_j + s_j + s_r \equiv 2n_j + 2n_{j-1} + 2n_r + 2n_{r-1} \equiv 0$, as required by Proposition \ref{prop:parametrize-irreps-mod-2}.

\color{black}
\begin{corollary}\label{cor:necessary-conditions}
For $(z,s) \in \Z^{r-1} \times \N^r$ and $n \in \N^r$,
let $b(z) \in \Z^r$ be the vector with entries $b_r = 0$ and $b_k = \sum_{i \in [k]} z_i - \frac{k}{r} \sum_{i \in [r-1]} z_i$ for $k \in [r-1]$. With $K=S(GL_2^r)$ we have
\begin{equation*}
    \text{dim Hom}_K (F_{z,s}, \mathcal{H}_n) > 0 \hspace{1cm} \implies \hspace{1cm} \begin{array}{l}
        \sum_{i \in [r-1]} z_i \equiv 0 \text{ mod r}, \\
        n \in b(z) + \N e, \\
        n_i + n_{i-1} \equiv s_i \text{ mod 2}, \forall i \in [r].
    \end{array}
\end{equation*}
For $K=GL_2^r$, $\text{dim Hom}_K (F_{z,s}, \mathcal{H}_n) > 0$ implies that $\sum_{i \in [r]} z_i = 0$, that $n_i + n_{i-1} \equiv s_i \text{ mod 2}$ for all $i \in [r]$, and that $n \in b(z) + \N e$ if we take $b(z) = (z_1,\dots,z_{r-1},0)$ instead.
\end{corollary}

\begin{proof}
To prove that $\sum_{i \in [r-1]} z_i = 0 \text{ mod r}$, we use that $z_i = n_i - n_{i-1} - n_r + n_{r-1}$ for all $F_{z,s}$ appearing in the decomposition of Theorem \ref{thm:harmonics-character}, and note that the sum partially telescopes, leaving $rn_{r-1} - r n_r$ which is zero mod r. To prove that $n_i + n_{i-1} = s_i \text{ mod 2}$ observe that $s \in \partial \lambda_m \implies s \in \lambda_m$ which implies that $s_i = m_i + m_{i-1} - 2\ell_i$ for some $\ell_i \in \mathbb{N}$. But since $m \in \partial \Lambda_n \implies m \in \Lambda_n$ then each $m_j = n_j - 2k_j$ for some $k_j \in \mathbb{N}$. Putting these together and taking equalities mod 2 yields the result. The fact that $n \in b(z) + \N e$ follows by solving the underdetermined linear system of equations $z_i = n_i - n_{i-1} - n_r + n_{r-1}$, solving for the $n_i, \, i \in [r]$ in terms of the $z_j, \, j \in [r-1]$ and taking $n_r$ as the free variable. The case $K = GL_2^r$ is similar.
\end{proof}

The second main result, Theorem \ref{thm:geometric-multiplicity}, describes the multiplicities of $F_{z,s}$ in $\mathcal{H}_n$ by counting integer points on certain polytopes. First, for $n \in \N^r$, denote by $P_n$ the polyhedron
\begin{equation*}
    P_n = \{ x \in \R^r \,\, : \,\, x_i \leq n_i, \,\, \forall i \in [r] \}.
\end{equation*}
Next, for $s \in \N^r$, denote by $Q_s$ the polyhedron
\begin{equation*}
    Q_s = \left\{ x \in \R^r \,\, : \,\, \begin{array}{c}
        \forall i \in [r]\\
        x_{i-1} + x_i - s_i \geq 0 \\
        x_{i-1} - x_i + s_i \geq 0 \\
        -x_{i-1} + x_i + s_i \geq 0
    \end{array}  \right\}.
\end{equation*}
Now, for any vector $v \in \R^r$ and any polyhedron $P \subset \R^r$, let $\partial_v P$ denote the faces of $P$ for which movement by any positive amount in the direction $v$ leaves the polyhedron. In particular, for $e = (1,1,\dots,1) \in \R^r$ we describe $\partial_e P_n$ and $\partial_{-e} Q_s$. For an affine-linear map $f:\R^r \to \R$, let $H_f$ denote the hyperplane defined by $f(x) = 0$, and for $i \in [r]$ let $f_i(x) = x_{i-1} + x_i - s_i$ and $g_i(x) = x_i - n_i$. Then we have
\begin{equation*}
    \partial_e P_n = \bigcup_{i \in [r]} H_{g_i} \cap P_n \hspace{1cm} \text{and} \hspace{1cm}
    \partial_{-e} Q_s = \bigcup_{i \in [r]} H_{f_i} \cap Q_s.
\end{equation*}
Intuitively, $\partial_e P_n$ is the union of the $r$ ``upper'' faces of $P_n$ while $\partial_{-e} Q_s$ is the union of the $r$ ``lower'' faces of $Q_s$, as measured by the direction $e \in \R^r$.  
See Figure \ref{fig:r=3-example-multiplicity-six} for an example of the intersection $\partial_e P_n \cap \partial_{-e} Q_s$.

\begin{figure}[h]
    \centering
    \includegraphics[width=0.45\textwidth]{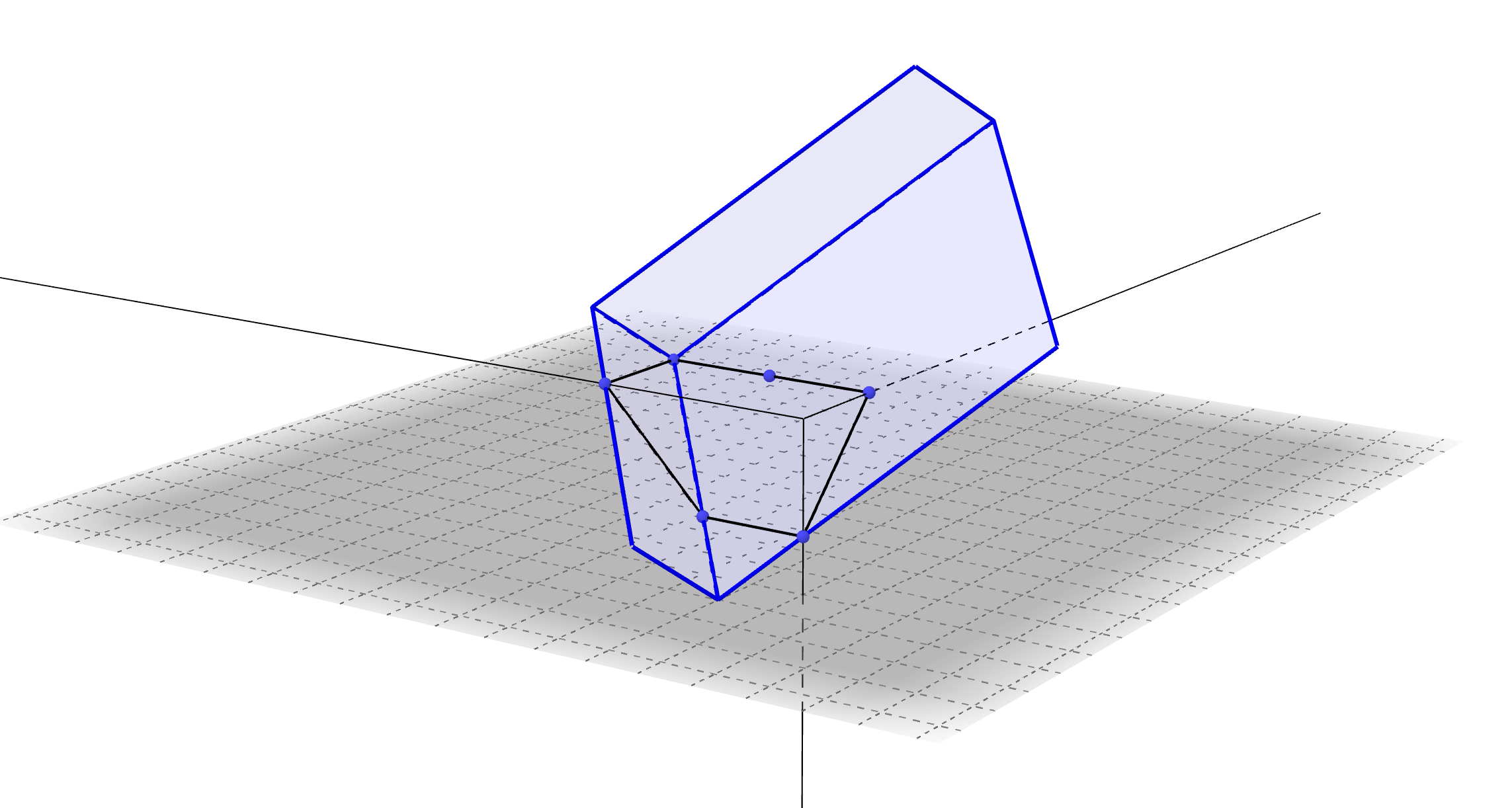}
    \includegraphics[width=0.45\textwidth]{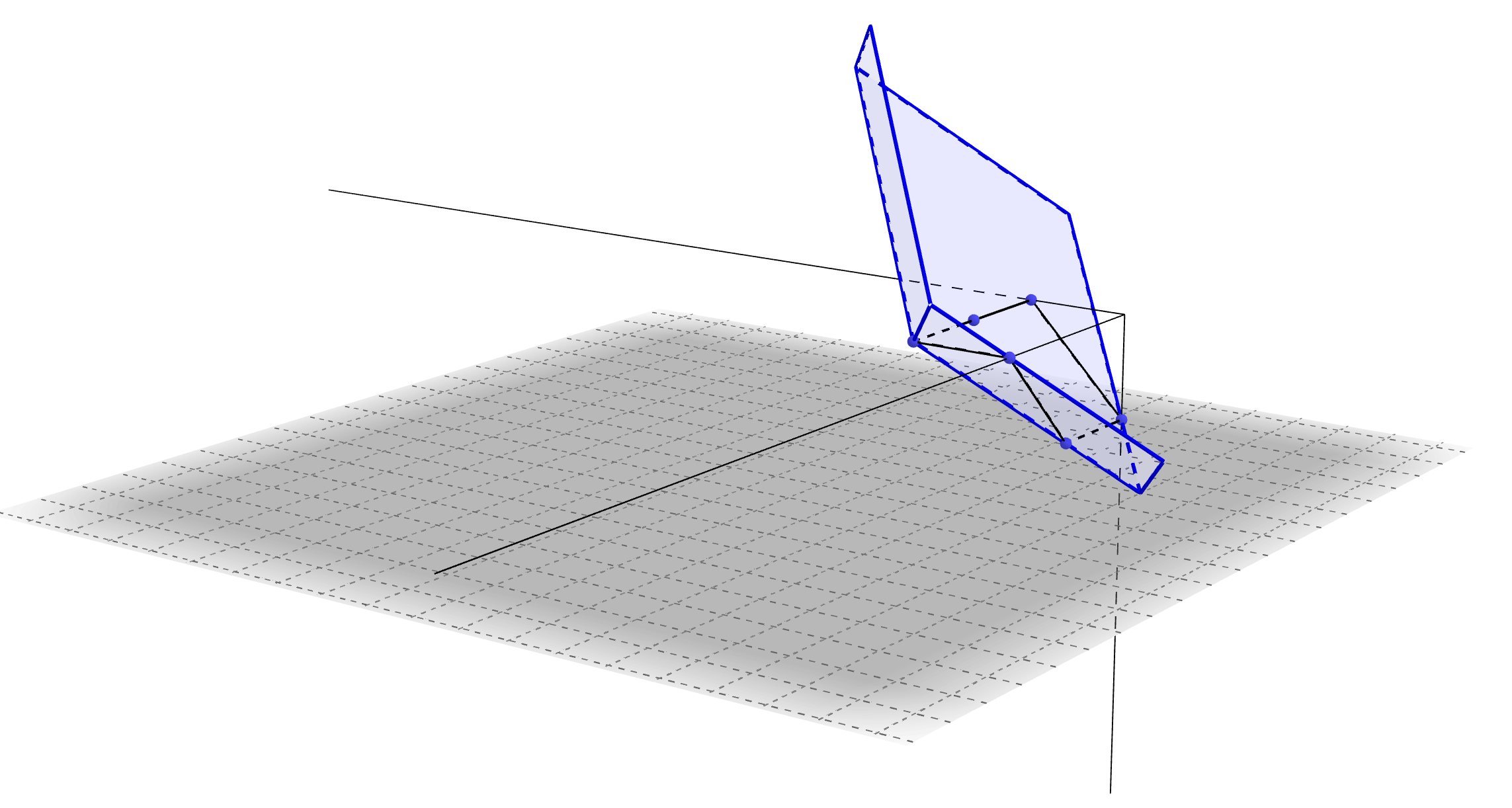}
    \caption{$\text{dim Hom}_K(F_{z,s},\mathcal{H}_n) = 6$ for $r=3$, $n=(6,5,3)$, $s = (7,5,4)$, and $z=(5,1)$.}
    \label{fig:r=3-example-multiplicity-six}
\end{figure}

\begin{theorem}\label{thm:geometric-multiplicity}
Let $K=GL_2^r$ or $K=S(GL_2^r)$. If $\text{dim Hom}_K (F_{z,s}, \mathcal{H}_n) > 0$, then
\begin{equation}\label{eq:count-points-on-polyhedra}
    \text{dim Hom}_K (F_{z,s}, \mathcal{H}_n)  = \# \left\{ m \in \N^r \,\, : \,\, \begin{array}{l}
        m \in \partial_e P_n \cap \partial_{-e} Q_s \\
        m_i = n_i \text{ mod 2}, \forall i \in [r] 
    \end{array} \right\}.
\end{equation}
In words, the multiplicity of $F_{z,s}$ in $\mathcal{H}_n$ is equal to the number of integer points on the intersection of the upper and lower faces of the polyhedra $P_n$ and $Q_s$, subject to a mod 2 condition.
\end{theorem}

\begin{proof}%[Proof of Theorem \ref{thm:geometric-multiplicity}]
Using Theorem \ref{thm:harmonics-character}, we need only verify that the set $A = \{ m \in \partial \Lambda_n : s \in \partial \lambda_m \}$ is equal to the set defined in the right side of (\ref{eq:count-points-on-polyhedra}), call it $B$. First, let $m \in A$ so that $m \in \partial \Lambda_n$ with $s \in \partial \lambda_m$. Since $m \in \partial \Lambda_n \subset \Lambda_n$ then $m_i = n_i \text{ mod 2}$ for all $i \in [r]$, and there exists an index $\ell \in [r]$ with $m_\ell = n_\ell$, so that also $m \in \partial_e P_n$. If we can show that $m \in \partial_{-e}Q_s$ then we have verified that $A \subset B$. Since $s \in \partial \lambda_m \subset \lambda_m$, each $s_i \in \{m_{i-1} + m_i, m_{i-1} + m_i - 2, \dots, |m_{i-1} - m_i| \}$ then $s_i \geq |m_{i-1} - m_i|$ which means that $m_{i-1} - m_i + s_i \geq 0$ and $-m_{i-1} + m_i + s_i \geq 0$. Since $s \in \partial \lambda_m$ we know there exists an index $k \in [r]$ such that $m_{k-1} + m_k - s_k = 0$. But then $m \in H_{f_k} \cap Q_s$, so that also $m \in \partial_{-e} Q_s$, and~$A \subset B$.

To finish, we need to show that $B \subset A$. Let $m \in B$ so that $m \in \N^r$ with $m_i = n_i \text{ mod 2}$ for all $i \in [r]$ and $m \in \partial_e P_n \cap \partial_{-e} Q_s$. Since $m \in P_n$ then $m_i \leq n_i$ for all $i \in [r]$ and since $m \in \partial_e P_n$ then there exists an index $\ell$ such that $m_\ell = n_\ell$. Therefore we have $m \in \partial \Lambda_n$. Showing $s \in \partial \lambda_m$ will complete the proof. Since $m \in Q_s$ then for all $i \in [r]$ we have $s_i \leq m_{i-1} + m_i$, $s_i \geq m_{i-1} - m_i$, and $s_i \geq m_i - m_{i-1}$. Since $m \in \partial_{-e} Q_s$ there exists an index $k \in [r]$ such that $m \in H_{f_k}$ and $m_{k-1} + m_k = s_k$. Finally, since $\text{dim Hom}_K(F_{z,s}, \mathcal{H}_n) > 0$, Corollary \ref{cor:necessary-conditions} implies that $n_{i-1} + n_i = s_i \text{ mod 2}$, but since $n_i = m_i \text{ mod 2}$ then also $m_{i-1} + m_i = s_i \text{ mod 2}$, so that $s \in \lambda_m$ and hence $s \in \partial \lambda_m$. Thus $m \in A$, completing the proof.
\end{proof}

\begin{corollary}
For any fixed $F_{z,s}$,
\begin{equation*}
    \text{dim Hom}_K(F_{z,s}, \mathcal{H}) = \sum_{n \in \N^r} \text{dim Hom}_K(F_{z,s}, \mathcal{H}_n) < \infty.
\end{equation*}
\end{corollary}
\begin{proof}
This result follows from the general theory and was calculated in \cite{W17} using an induced representation. We simply note that the geometry in Theorem \ref{thm:geometric-multiplicity} confirms it. As $n$ moves up along the ray $b(z) + \N e$, eventually $\partial_{-e} Q_s \subset \text{int} P_n$, and hence the intersection with $\partial_e P_n$ is empty past that point.
\end{proof}

\color{black}

\section{Proof of Theorem \ref{thm:harmonics-character}}\label{sec:proofs}

In this section we prove Theorem \ref{thm:harmonics-character}. We first examine $\C[V]_n$ and $\Harm_n$ as representations of $\Tilde{K}$, and then observe that in the irreducible representations that appear, $N = \ker \varphi$ acts by the identity (recall the discussion immediately following Theorem \ref{thm:harmonics-character}). Recall $\Tilde{K} = (\C^\times)^\ell \times SL_2^r$ for $\ell=r-1$ or $\ell=r$. Let $Z^\ell \simeq (\C^\times)^\ell$ denote the subgroup of $\Tilde{K}$ obtained by choosing the identity element in each factor of $SL_2$. Let $T^r$ denote the subgroup isomorphic to $(\C^\times)^r$
\begin{equation*}
    \{ u = \left(\diag(u_1,u_1^{-1}), \dots, \diag(u_r,u_r^{-1}) \right) \,\, : \,\, u_j \in \C^\times, \forall j \in [r] \}
\end{equation*}
obtained by choosing $1 \in \C^\times$ in each of the factors of $(\C^\times)^\ell$ and diagonal matrices in the $SL_2^r$ factors. Notice that $T^r$ may also be identified with $\varphi(T^r)$ as a subgroup of $K$ in both cases. We will call the elements $u \in T^r$, where $u = \left(\diag(u_1,u_1^{-1}), \dots, \diag(u_r,u_r^{-1}) \right)$, but also refer to the components $u_i$ for $i \in [r]$. \color{black} We first examine the action of $Z^\ell$ on $\C[V]_n$. \color{black}

\begin{proposition}\label{prop:central-characters}
For $n \in \N^r$, consider $\C[V]_n$ as a representation of $\Tilde{K}$. The action of $w \in Z^r$ is given by $f \mapsto \left(\prod_{i \in [r]} w_i^{n_i - n_{i-1}}\right) f$ for all $f \in \C[V]_n$, where $Z^r$ is the subgroup of $\Tilde{K} = (\C^\times)^r \times SL_2^r$. The action of $w \in Z^{r-1}$ is given by $f \mapsto \left( \prod_{i \in [r-1]} w_i^{n_i - n_{i-1} - n_r + n_{r-1}} \right) f$ for all $f \in \C[V]_n$, where $Z^{r-1}$ is the subgroup of $\Tilde{K} = (\C^\times)^{r-1} \times SL_2^r$.
\end{proposition}

\begin{proof}
This follows from Equation (\ref{eqn:action-by-cyclic-conjugation}) which says that $X_i \mapsto g_i X_i g_{i+1}^{-1}$ for each $(X_1,\dots,X_r) \in V$. Note that because of the formula $f(g^{-1}x)$ the inverse switches places for the action on $\C[V]$. By definition of $\varphi:\Tilde{K} \to K$, the only change from $Z^r$ to $Z^{r-1}$ is that we replace $w_r$ with $w_1^{-1}w_2^{-1}\cdots w_{r-1}^{-1}$ which multiplies by additional $w_i^{n_{r-1} - n_r}$ factors.
\end{proof}

We now focus on the action of the subgroup $T^r = \varphi(T^r) \subset K$, which is the same in both cases $K=GL_2^r$ and $K=S(GL_2^r)$. For a representation $\rho$ of $T^r$ on a vector space $W$, let $\chi(W)$ denote $tr(\rho(u))$, which we call the character of $T^r$ on $W$ and is a function of the variables $u_1,\dots,u_r$ from $u \in T^r$. In other words, we have $\chi(W):T^r \to \C$. For $\ell \in \N$ and $j \in [r]$ let $\chi_\ell(u_j)$ denote the expression $\chi_\ell(u_j) = u_j^\ell + u_j^{\ell - 2} + \cdots + u_j^{-\ell}$, which will appear frequently. Recall \color{black} that $Z^\ell \times T^r$ is a maximal torus in $\Tilde{K} = (\C^\times)^\ell \times SL_2^r$ and any completely reducible representation of $\Tilde{K}$ on $W$ is determined up to isomorphism by $\chi(W)$, see \cite[page 188]{GoodmanWallach}.

\begin{proposition}\label{prop:invariants}\color{black} For $n \in \N^r$, let $n_{min} = \text{min}\{ n_i : i \in [r]\}$ and $e = (1,\dots,1) \in \N^r$. For $\rho$ a representation of $T^r$ and $W$ a subrepresentation, let $\chi(W) = tr(\rho(u)|_W)$ for $u \in T^r$. Restrict the representations of $K$ on $\C[V]_n$ and $\Harm_m$ to $T^r$, for each $n,m \in \N^r$. \color{black} Then for all $n \in \N^r$ we have
\begin{equation}\label{eqn:invariant-march}
    \chi(\C[V]_n) = \sum_{j\in \{0,1,\dots,n_{min}\}} \big( \lfloor j/2 \rfloor + 1 \big) \chi(\mathcal{H}_{n-je}).
\end{equation}
If $n_{min}$ is large, the formula begins
\begin{multline*}
    \chi(\C[V]_{n}) = \chi(\Harm_{n}) + \chi(\Harm_{n-e}) + 2\chi(\Harm_{n-2e}) + 2\chi(\Harm_{n-3e})\\
    + 3\chi(\Harm_{n-4e}) + 3\chi(\Harm_{n-5e}) + 4\chi(\Harm_{n-6e}) + 4\chi(\Harm_{n-7e}) + \dots 
\end{multline*}
\end{proposition}

\begin{proof}
By a result from \cite{LP90} we know that the ring of invariants $\C[V]^K$ is a polynomial algebra generated by $f=tr(X_1 X_2 \cdots X_r)$ and $g=tr((X_1 X_2 \cdots X_r)^2)$ which have multidegree $e$ and $2e$ respectively, so that $\C[V]^K = \C[f,g]$. \color{black}
Using $q$ to encode the total degree in $\C[V]$ and $u$ to encode the total degree in $\mathcal{H}$, the formula $\C[V] = \C[V]^K \otimes \mathcal{H}$ becomes
\begin{equation*}
    \Phi(q,u) = \frac{1}{(1-q)(1-q^2)(1-uq)}
\end{equation*}
combinatorially. For $n \in \N^r$, let $\ell \in [r]$ be an index such that $n_\ell \leq n_j$ for all $j \neq \ell$. Then the coefficient of $q^{n_\ell}$ in the series expansion of $\Phi(q,u)$ is a polynomial in $u$ that records the multiplicities of $\chi(\Harm_{n-je})$ inside $\chi(\C[V]_n)$ from Equation (\ref{eqn:invariant-march}).

\color{black} We explain in more detail. Let $h_1,\dots,h_a$ be a basis of $\Harm_{n - e}$. Then $fh_1,\dots, fh_a$ span a subspace of $\C[V]_n$ whose character contributes a term identical to $\chi(\Harm_{n-e})$ to the character $\chi(\C[V]_n)$, since $f \in \C[V]^K$. Similarly, if $\{h_j\}$ is a basis for $\Harm_{n - 2e}$ then $\chi(\text{span}\{ f^2 h_j \}) = \chi(\Harm_{n - 2e})$ and $\chi(\text{span}\{ g h_j \}) = \chi(\Harm_{n - 2e})$ also appear as terms inside $\chi(\C[V]_n)$. Since $f$ and $g$ have evenly distributed multidegrees $e$ and $2e$, the only $\chi(\Harm_m)$ that appear as terms within $\chi(\C[V]_n)$ are those with multidegree $n,n-e,n-2e,n-3e,\dots,n-n_{min}$, with varying multiplicities. To count those multiplicities we use the generating function $\Phi(q,u)$, where
\begin{equation*}
    \Phi(q,u) = \frac{1}{(1-q)(1-q^2)(1-uq)}.
\end{equation*}
Expanding $\Phi(q,u)$ as a series in $q$, the coefficients are polynomials in $u$. For instance, the coefficient $u^{n_\ell}$ of $q^{n_\ell}$ in the series expansion of $1/(1-uq)$ counts the term $\chi(\Harm_n)$, which appears with multiplicity one in $\chi(\C[V]_n)$. Multiplying $1/(1-uq)$ by $1/(1-q) = 1 + q + q^2 + q^3 + \cdots$ will count the additional terms $\chi(\Harm_{n-je})$ that appear due to powers of $f,f^2,f^3,\dots$. Multiplying $1/(1-uq)$ by $1/((1-q^2)(1-q)) = (1 + q^2 + q^4 + \cdots)(1 + q + q^2 + q^3 + \cdots)$ counts the additional terms appearing due to $f,g,f^2,fg,g^2,f^3,f^2g,fg^2,g^3,\dots$. Therefore, the coefficient of $q^{n_\ell}$ in the series expansion of $\Phi(q,u)$ is a polynomial in $u$ recording the multiplicities of the various $\chi(\Harm_m)$ for $m = n, n-e, n-2e,\dots, n-n_{min}e$ appearing in $\chi(\C[V]_n)$, since $\C[V] = \C[V]^K \otimes \Harm = \C[f,g] \otimes \Harm$. 
In general, the coefficient of $u^{n_\ell - j}$ inside the coefficient of $q^{n_\ell}$ equals $\lfloor j/2 \rfloor + 1$. \color{black}
\end{proof}

\begin{proposition}\label{prop:long-character}
\color{black} For $k \in \Z$ let $[k]_2 = 0$ if $k$ is even, and $[k]_2=1$ if $k$ is odd. For $\ell \in \N$ and $j \in [r]$, let $\chi_\ell(u_j) = u_j^\ell + u_j^{\ell - 2} + \cdots + u_j^{-\ell}$. Then for $n \in \N^r$, \color{black} the character $\chi(\C[V]_n)$ is
\begin{equation*}
    \prod_{i \in [r]} \bigg[ \chi_{n_i}(u_i)\chi_{n_i}(u_{i+1}) + \chi_{n_i-2}(u_i)\chi_{n_i-2}(u_{i+1}) + \cdots + \chi_{[n_i]_2}(u_i)\chi_{[n_i]_2}(u_{i+1}) \bigg].
\end{equation*}
\end{proposition}

\begin{proof}\color{black}
For $n \in \N^r$ let $n_i e_i = (0, \dots, 0, n_i, 0, \dots, 0)$. The multigraded component $\C[V]_n$ is a tensor product $\C[V]_n = \bigotimes_{i \in [r]} \C[V]_{n_i e_i}$ by \cite[Proposition C.1.4]{GoodmanWallach}.
Therefore we first compute the character of $\C[V]_{n_i e_i}$ and the result will follow by taking the product over $i \in [r]$.

\color{black} From the decomposition $V=\oplus M_i$, the space $\C[V]_{n_i e_i}$ consists of polynomials on $M_i$ alone, with the action of a group element $(g_1,g_2,\dots,g_r) \in K$ given by Equation (\ref{eqn:action-by-cyclic-conjugation}). Therefore the action of $K$ on $\C[V]_{n_i e_i}$ depends only on $g_i \in GL_2$ and $g_{i+1} \in GL_2$, and is identical to the usual action of $GL_2 \times GL_2$ on two by two matrices. Applying the well-known \cite[Theorem 5.6.7]{GoodmanWallach} and then restricting to $T^r$ we find  
\begin{equation*}
    \chi(\C[V]_{n_i e_i}) = \chi_{n_i}(u_i)\chi_{n_i}(u_{i+1}) + \chi_{n_i-2}(u_i)\chi_{n_i-2}(u_{i+1}) + \cdots + \chi_{[n_i]_2}(u_i)\chi_{[n_i]_2}(u_{i+1}).
\end{equation*}
The result follows by taking the product over $i \in [r]$.
\end{proof}

\color{black}
\begin{proposition}
For any $n \in \N^r$ we have
\begin{equation*}
    \chi(\C[V]_n) = \sum_{m \in \Lambda_n} \prod_{i \in [r]} \chi_{m_{i-1}}(u_i) \chi_{m_i}(u_i).
\end{equation*}
\end{proposition}

\begin{proof}
The terms in each factor of the product in Proposition \ref{prop:long-character} are indexed by $\{ n_i, n_i-2, \dots, [n_i]_2 \}$ and so when we expand the product into a sum of terms, the resulting terms will be indexed by points $m \in \Lambda_n$. Each of these terms is a product of one term from each of the $r$ factors. Each term looks like $\chi_{m_i}(u_i)\chi_{m_i}(u_{i+1})$. Thus, in the resulting product there will be two factors involving each $u_i$, but coming from neighboring indices $m_{i-1}$ and $m_i$.
\end{proof}

\begin{proposition}
    For any $n \in \N^r$ we have
    \begin{equation*}
        \chi(\C[V]_n) = \sum_{m \in \Lambda_n} \, \sum_{p \in \lambda_m} \chi_p,
    \end{equation*}
    where $\chi_p$ denotes the product $\chi_{p_1}(u_1)\chi_{p_2}(u_2)\cdots \chi_{p_r}(u_r)$ for $p \in \N^r$.
\end{proposition}
\begin{proof}
    We first consider $\prod_{i=1}^r \chi_{m_{i-1}}(u_i)\chi_{m_i}(u_i)$ from the previous proposition. By the well-known Clebsch-Gordan identity for decomposing tensor products of $SL_2$ irreducibles \cite[page 339]{GoodmanWallach}, we can multiply two characters and obtain a sum of irreducible characters, as in
    \begin{equation*}
        \chi_a(u)\chi_b(u) = \chi_{a+b}(u) + \chi_{a+b-2}(u) + \cdots + \chi_{|a-b|}(u).
    \end{equation*}
    Applying this to every factor in the product from the previous proposition we obtain
    \begin{equation*}
        \prod_{i=1}^r \chi_{m_{i-1}}(u_i)\chi_{m_i}(u_i) =
         \prod_{i=1}^r \bigg[ \chi_{(m_{i-1}+m_i)}(u_i) + \chi_{(m_{i-1}+m_i-2)}(u_i) + \cdots + \chi_{|m_{i-1}-m_i|}(u_i) \bigg].
    \end{equation*}
    But now it's clear that when we expand \textit{that} product, the terms in the resulting sum will be indexed by points $p \in \lambda_m$.
\end{proof}

\begin{lemma}\label{lemma:shell-decompositions}
For $n \in \N^r$, let $n_{min}$ denote $\text{min}\{ n_i : i \in [r] \}$, and let $\Lambda_n$ and $\lambda_n$ be as described in Theorem \ref{thm:harmonics-character}, with $\partial \Lambda_n$ denoting the subset of all $x \in \Lambda_n$ such that $x+2e \notin \Lambda_n$ and similarly for $\partial \lambda_n$. Then
\begin{align*}
    \Lambda_n &= \partial \Lambda_n \cup \Lambda_{n-2e} \text{ if } n_{min} \geq 2,\\
    \Lambda_n &= \partial \Lambda_n \text{ if } 0 \leq n_{min} < 2,\\
    \lambda_m &= \partial \lambda_m \cup \partial \lambda_{m-e} \cup \cdots \cup \partial \lambda_{m - m_{min}e}.
\end{align*}
\end{lemma}

\begin{proof}
These facts follow easily from the definitions. We note that replacing $m$ by $m - e$ we have $(m_{i-1} - 1) + (m_i - 1) = m_{i-1} + m_i - 2$ but $|(m_{i-1} - 1) - (m_i - 1) | = |m_{i-1} - m_i|$ and therefore $\lambda_m = \partial \lambda_m \cup \lambda_{m-e}$. Applying this repeatedly yields the result.
\end{proof}

\color{black} Recall the sets $\Lambda_n$ and $\lambda_n$ were defined for $n \in \mathbb{N}^r$. However, we now introduce the convention that $\Lambda_n, \partial \Lambda_n, \lambda_n,$ and $\partial \lambda_n$ are empty sets whenever $n \in \mathbb{Z}^r \setminus \mathbb{N}^r$, i.e. whenever $n$ has at least one negative component, and the further convention that sums over empty sets are zero. \color{black}

\begin{proposition} \label{prop:beforeReindex} 
For any $n \in \mathbb{N}^r$ such that $n_{min} \geq 2$ we have
    \begin{equation*}
        \chi(\C[V]_n) - \chi(\C[V]_{n-2e}) =
        \sum_{m \in \partial \Lambda_n} \bigg[  \sum_{p \in \partial \lambda_m} \chi_p  +  \sum_{p \in \partial \lambda_{m-e}} \chi_p  + \cdots +  \sum_{p \in \partial \lambda_{m - n_{min}e}} \chi_p  \bigg].
    \end{equation*}
\end{proposition}

\begin{proof}
    Since $\Lambda_n = \partial\Lambda_n \cup \Lambda_{n-2e}$ whenever $n_{min} \geq 2$ by Lemma \ref{lemma:shell-decompositions}, we have
    \begin{align*}
        \chi(\C[V]_n) &= \sum_{m \in \Lambda_n} \sum_{p \in \lambda_{m}} \chi_{p} \\
         &= \sum_{m \in \partial\Lambda_n}  \sum_{p \in \lambda_{m}} \chi_{p}  + \sum_{m \in \Lambda_{n-2e}}  \sum_{p \in \lambda_{m}} \chi_{p} 
    \end{align*}
    But the second term is simply $\chi(\C[V]_{n-2e})$, which means
    \begin{equation*}
        \chi(\C[V]_n) = \sum_{m \in \partial\Lambda_n}  \sum_{p \in \lambda_{m}} \chi_{p}  + \chi(\C[V]_{n-2e})
    \end{equation*}
    Since by Lemma \ref{lemma:shell-decompositions}
    \begin{equation*}
        \lambda_{m} = \partial \lambda_{m} \cup \partial \lambda_{m-e} \cup \cdots \cup \partial \lambda_{m - m_{min}e}
    \end{equation*}
    we have
    \begin{equation*}
        \chi(\C[V]_n) - \chi(\C[V]_{n-2e}) = 
        \sum_{m \in \partial\Lambda_n} \bigg[  \sum_{p \in \partial \lambda_{m}} \chi_{p}  +  \sum_{p \in \partial \lambda_{m-e}} \chi_{p}  + \cdots + \sum_{p \in \partial \lambda_{m - m_{min}e}} \chi_{p} \bigg]
    \end{equation*}
    \color{black} The expression in brackets has $m_{min} + 1$ sums, which obviously depends on the choice of $m \in \partial \Lambda_n$. Since $n \in \partial \Lambda_n$, we may have $m=n$, which is when $m_{min}$ is maximized. Thus we can replace $m - m_{min}e$ with $m - n_{min}e$ in the expression in brackets, with the convention that sums over $p \in \partial \lambda_{m-je}$ are empty when $j > m_{min}$. This introduces empty sums, but simplifies the expression in a way that becomes useful during Propositions \ref{prop:reindex} and \ref{prop:finalinduction}. \color{black}
\end{proof}

Now consider distributing the outer sum over the inner ones, and examine just one term.
\begin{proposition} \label{prop:reindex} Re-indexing the double sum: For any $n \in \mathbb{N}^r$ and $j \in \mathbb{N}$ we have
    \begin{equation*}
        \sum_{m \in \partial\Lambda_n}  \sum_{p \in \partial \lambda_{m-je}} \chi_{p}  = \sum_{m' \in \partial \Lambda_{n-je}}  \sum_{p \in \partial \lambda_{m'}} \chi_{p} 
    \end{equation*}    
\end{proposition}

\begin{proof}\color{black} 
We need only verify the equality of the indexing sets $A = \{ p : p \in \partial\lambda_{m - je} \text{ for some } m \in \partial \Lambda_n \}$ and $B = \{ p : p \in \partial\lambda_{m'} \text{ for some } m' \in \partial\Lambda_{n - je} \}$.

First we prove $A \subseteq B$. \color{black} Choose $p \in \partial \lambda_{m-je}$ for some $m \in \partial \Lambda_n$. In particular,  $\partial \lambda_{m-je} \neq\emptyset$, so $j \leq m_{min}$. Define $m'$ by defining $m_i' = m_i - j$ for all $i$. We have that $p_i \in \{ m_{i-1}-j+m_i-j, m_{i-1}-j+m_i-j-2, \dots, |m_{i-1}-j-(m_i-j)| \}$ for all $i$, where there is some $k$ such that $p_k = m_{k-1}-j+m_k-j$, since we are in the $ \partial \lambda_{m-je}$. But this is the same as saying that $p_i \in \{ m_{i-1}'+m_i', m_{i-1}'+m_i'-2, \dots, |m_{i-1}'-m_i')| \}$ for all $i$, and for the same $k$, $p_k = m_{k-1}'+m_k'$. Thus $p \in \partial \lambda_{m'}$ as needed. Now, is $m' \in \partial \Lambda_{n-je}$? Since $m \in \partial\Lambda_n$ then $m_i \in \{ n_i, n_i-2, \dots, [n_i]_2 \}$ for all $i$, and there is some $k$ such that $m_k=n_k$. Since $j \leq m_{min}$ then $m_i' = m_i - j$ is always $\geq 0$ for all $i$, so we know $m_i' \in \{ n_i-j, n_i-j-2, \dots, [n_i-j]_2 \}$ for all $i$, and for the same $k$, $m_k'+j=n_i$ so $m_k' = n_i-j$. This means $m' \in \partial \Lambda_{n-je}$ as required.

We now prove $A \supseteq B$. Pick $p \in \partial \lambda_{m'}$ for some $m' \in \partial \Lambda_{n-je}$. Define $m$ by defining $m_i = m_i'+j$ for all $i$. Is $m \in \partial \Lambda_n$? We have that $m_i' \in \{ n_i-j, n_i-j-2, \dots, [n_i-j]_2 \}$ for all $i$, which means that $m_i'+j \in \{ n_i, n_i-2, \dots, [n_i-j]_2 +j \}$, which implies that $m_i \in \{ n_i, n_i-2, \dots, [n_i]_2 \}$ for all $i$. We also have the existence of $l$ such that $m_l'=n_l-j$, implying $m_l-j=n_l-j$, which means $m_l=n_l$. Thus $m \in \partial\Lambda_n$. Is $p \in \partial \lambda_{m-je}$? We have that $p_i \in \{ m_{i-1}'+m_i', m_{i-1}'+m_i'-2, \dots, |m_{i-1}'-m_i'| \}$ for all $i$, where for some $k$ we have $p_k = m_{k-1}'+m_k'$. Then, for all $i$, since $m_i = m_i'+j$ we have $p_i \in \{ m_{i-1}-j+m_i-j, m_{i-1}-j+m_i-j-2, \dots, |(m_{i-1}-j)+(m_i-j)| \}$ with $p_k=m_{k-1}-j+m_k-j$. This means $p \in \partial \lambda_{m-je}$. This completes the proof.
\end{proof}

\begin{proposition} \label{prop:finalinduction} For any $n \in \mathbb{N}^r$ such that $n_{min} \geq 2$ we have
    \begin{align*}
        \chi(\C[V]_n) - \chi(\C[V]_{n-2e}) &= \sum_{m \in \partial \Lambda_n}  \sum_{p \in \partial \lambda_{m} } \chi_{p} \\
         & + \sum_{m \in \partial \Lambda_{n-e} }  \sum_{p \in \partial \lambda_{m}} \chi_{p} \\
         & + \cdots \\
         & + \sum_{m \in \partial \Lambda_{n-n_{min}e}}  \sum_{p \in \partial \lambda_{m}} \chi_{p} 
    \end{align*}
\end{proposition}

\begin{proof}
    This follows easily from the previous two Propositions. First distribute the outer sum over the inner sums in Proposition \ref{prop:beforeReindex} and then re-index each resulting double sum by using Proposition \ref{prop:reindex}.
\end{proof}

Finally we come to the decomposition of $\mathcal{H}_n$ and the proof of Theorem \ref{thm:harmonics-character}.

\begin{proof}[Proof of Theorem 1]
Since Equation (\ref{eq:harmonics-character}) requires that we find
\begin{equation*}
    \mathcal{H}_n = \bigoplus_{m \in \partial \Lambda_n} \bigoplus_{p \in \partial \lambda_m} F_{z,p}
\end{equation*}
we will prove the theorem by examining the character $\chi(\mathcal{H}_n)$. Recall that for $\ell \in \N$ and $j \in [r]$, we let $\chi_\ell(u_j) = u_j^\ell + u_j^{\ell - 2} + \cdots + u_j^{-\ell}$. By Proposition \ref{prop:central-characters} we already know \color{black} the action of $Z^\ell$ is correct in both cases, so we need only verify that for $T^r$ we have \color{black}
\begin{equation*}
        \chi \big( \mathcal{H}_n \big) = \sum_{m \in \partial\Lambda_n}   \sum_{p \in \partial \lambda_{m}} \chi_{p}
\end{equation*}
where the $\chi_{p}$ for $p \in \N^r$ denotes the product
\begin{equation*}
    \chi_{p_1}(u_1)\chi_{p_2}(u_2) \cdots \chi_{p_r}(u_r).
\end{equation*}

The proof is by induction. Consider the graded component $n \in \N^r$. We refer to $\text{min}\{ n_i : i \in [r] \}$ as $n_{min}$. The base case is when $n_{min}=0$. In this case, $\Lambda_n = \partial\Lambda_n$, since if $n_k=0$ then $\{ n_k, n_k-2, \dots, [n_k]_2 \}$ degenerates to  $\{ 0 \}$. Then every $m \in \partial\Lambda_n$ has at least one component, say $k$, where $m_k=0$. Recall $\lambda_{m} = \prod \{ m_{i-1}+m_i, m_{i-1}+m_i-2, \dots, |m_{i-1}-m_i| \}$ by definition, but then the $k$th factor becomes  $\{ m_{k-1} \}$. This implies that $\partial \lambda_{m}$ = $\lambda_{m}$. Then
\begin{equation*}
    \chi(\C[V]_n) = \sum_{m \in \Lambda_n}  \sum_{p \in \lambda_{m}} \chi_{p}  = \sum_{m \in \partial \Lambda_n}   \sum_{p \in \partial \lambda_{m}} \chi_{p}.
\end{equation*}\color{black}
Note that $\chi(\mathbb{C}[V]_n) = \chi(\mathcal{H}_n)$ in the case $n_{min}=0$, since the two generators $f=tr(X_1 X_2 \cdots X_r)$ and $g=tr((X_1 X_2 \cdots X_r)^2)$ of $\mathbb{C}[V]^K$ have multidegrees $(1,1,\dots,1)$ and $(2,2,\dots,2)$, respectively. Since $\chi(\C[V]_n) = \chi(\mathcal{H}_n)$ has the appropriate formula, that completes the base case $n_{min} = 0$.

Consider now the case $n_{min} = 1$. Then $\C[V]_n$ contains one of the generators for the invariants, $f$, of multigraded degree $e = (1,1,\dots,1)$, and does not contain the second generator $g$. \color{black} In this case
\begin{equation*}
    \chi(\C[V]_n) = \chi(\mathcal{H}_n) + \chi(\Harm_{n-e})
\end{equation*}
At this point, we know both $\chi(\C[V]_n)$ and $\chi(\Harm_{n-e})$, and we can solve for the unknown $\chi(\mathcal{H}_n)$. Since $n_{min} = 1$ implies $\partial \Lambda_n$ = $\Lambda_n$, we have
\begin{align*}
    \chi(\C[V]_n) &= \sum_{m \in \Lambda_n}  \sum_{p \in \lambda_{m}} \chi_{p}  \\
     &= \sum_{m \in \partial\Lambda_n} \bigg[  \sum_{p \in \partial \lambda_{m}} \chi_{p} + \sum_{p \in \partial \lambda_{m-e}} \chi_{p} \bigg] \\
      &= \sum_{m \in \partial\Lambda_n}  \sum_{p \in \partial \lambda_{m}} \chi_{p}  + \sum_{m \in \partial\Lambda_n}  \sum_{p \in \partial \lambda_{m-e}} \chi_{p} \\
       &= \sum_{m \in \partial\Lambda_n}  \sum_{p \in \partial \lambda_{m}} \chi_{p}  + \sum_{m \in \partial \Lambda_{n-e}}  \sum_{p \in \partial \lambda_{m}} \chi_{p} \\
       &= \sum_{m \in \partial\Lambda_n} \sum_{p \in \partial \lambda_{m}} \chi_{p} + \chi(\Harm_{n-e})
\end{align*}
where we have used Proposition \ref{prop:reindex}, as well as our induction base case of $n_{min}=0$. It is now clear that $\chi(\mathcal{H}_n)$ is leftover, and has the correct formula.

Consider the graded component $n \in \N^r$, where $n_{min} \geq 2$. Since by induction we know the character of any $\Harm_{n-je}$ for $j > 0$ then Proposition \ref{prop:finalinduction} becomes
\begin{equation*}
    \chi(\C[V]_n) - \chi(\C[V]_{n-2e}) = \sum_{m \in \partial\Lambda_n} \sum_{p \in \partial \lambda_{m}} \chi_{p} + \chi(\Harm_{n-e}) + \dots + \chi(\Harm_{n-n_{min}e})
\end{equation*}
We can break up Proposition \ref{prop:invariants} into two separate lines as in the table below.
\begin{table}[h] \label{tab:subtraction}
    \centering
    \scalemath{0.9}{
    \begin{tabular}{c c c c c c c c }
    $\chi(\mathcal{H}_n)$ & $=\chi(\C[V]_n)$ & $- \chi(\Harm_{n-e})$ & $-\chi(\Harm_{n-2e})$ & $-\chi(\Harm_{n-3e})$ & $-\chi(\Harm_{n-4e})$ & $- \dots$  \\
      &  &  & $-\chi(\Harm_{n-2e})$ & $- \chi(\Harm_{n-3e})$ & $- 2\chi(\Harm_{n-4e})$ & $- \dots$ \\
     \hline
    $\chi(\mathcal{H}_n)$ & $=\chi(\C[V]_n)$ & $- \chi(\Harm_{n-e})$ & $-2\chi(\Harm_{n-2e})$ & $- 2\chi(\Harm_{n-3e})$ & $- 3\chi(\Harm_{n-4e})$ & $- \dots$
    \end{tabular}}
    \caption{Subtracting characters in two ways}
\end{table}
But recognizing the second line of subtractions in the table as nothing but $\chi(\C[V]_{n-2e})$, and comparing the two equations, we see that
\begin{equation*}
    \chi \big( \mathcal{H}_n \big) = \sum_{m \in \partial\Lambda_n}   \sum_{p \in \partial \lambda_{m}} \chi_{p} 
\end{equation*}
as claimed. This completes the proof of Theorem \ref{thm:harmonics-character}.
\end{proof}

% \bibliographystyle{amsplain}
% \bibliography{references}

 {\footnotesize\linespread{0.8}
\bibliographystyle{amsplain}
\bibliography{references}

Alexander Heaton, Lawrence University, Appleton, WI, USA
}

%\bibliographystyle{plain}
%\bibliographystyle{alpha}
%\bibliography{references}

\end{document}